\documentclass{article}%
\usepackage{amsmath}
\usepackage{amsfonts}
\usepackage{amssymb}
\usepackage{graphicx}%
\providecommand{\U}[1]{\protect\rule{.1in}{.1in}}
\newtheorem{theorem}{Theorem}[section]
\newtheorem{acknowledgement}[theorem]{Acknowledgement}

\newtheorem{corollary}[theorem]{Corollary}

\newtheorem{definition}[theorem]{Definition}

\newtheorem{lemma}[theorem]{Lemma}

\newtheorem{proposition}[theorem]{Proposition}
\newtheorem{remark}[theorem]{Remark}

\newenvironment{proof}[1][Proof]{\noindent\textbf{#1.} }{\ \rule{0.5em}{0.5em}}
\begin{document}

\title{Leavitt path algebras with finitely presented irreducible representations }
\author{Kulumani M. Rangaswamy\thanks{
graphs, simple modules, finitely presented modules, Gelfand-Kirillov
dimension.}\\Department of Mathematics, University of Colorado \\Colorado Springs, Colorado 80918, USA\\E-mail: krangasw@uccs.edu}
\date{}
\maketitle

\begin{abstract}
Let $E$ be an arbitrary graph, $K$ be any field and let $L=L_{K}(E)$ be the
corresponding Leavitt path algebra. Necessary and sufficient conditions (both
graphical and algebraic) are given under which all the irreducible
representations of $L$ are finitely presented. In this \ case, the graph $E$
\ turns out to be row-finite and the cycles in $E$ form an artinian partial
ordered set under a defined relation $\geq$. When the graph is $E$ is finite,
the above graphical conditions were shown in \cite{AAJZ1} to be equivalent to
$L_{K}(E)$ having finite Gelfand-Kirillov dimension. Examples show that this
equivalence no longer holds for infinite graphs and a complete description is
obtained of Leavitt path algebras over arbitrary graphs having finite
Gelfand-Kirillov dimensions.

\end{abstract}

\section{Introduction and Preliminaries}

The notion of Leavitt path algebras was introduced and initially studied in
\cite{AA}, \cite{AMP} as algebraic analogues of graph C$^{\ast}$-algebras and
as the natural generalization of the Leavitt algebras of type (1,n) built in
\cite{Le}. The module theory over Leavitt path algebras was initiated in
\cite{AB} and in other recent papers (\cite{AR1}, \cite{AR2}, \cite{R}). In
\cite{GR}, Goncalves and Royer indicated a method of constructing various
representations of a Leavitt path algebra $L_{K}(E)$ over a graph $E$ by using
the concept of algebraic branching systems. Expanding this, Chen \cite{C}
studied special types of irreducible representations of $L_{K}(E)$ induced by
the sinks as well as the equivalence class $[p]$ of infinite paths
tail-equivalent (see definition below) to a fixed infinite path $p$ in $E$ and
he further noted that these can also be considered as algebraic branching
systems. Additional ways of constructing irreducible representations of
$L_{K}(E)$ were pointed out in \cite{AR1} while in \cite{R} a new class
$\mathbf{S}_{v}$ of irreducible representations was constructed using vertices
$v$ which emit infinitely many edges.

When $E$ is a finite graph, it was shown in \cite{AR1} that every simple left
module over $L_{K}(E)$ is finitely presented if and only if every vertex in
$E$ is the base of at most one cycle. In this paper, we wish to extend this
theorem to the case when $E$ is an arbitrary graph. Unlike the case of finite
graphs, the existence of infinite paths and vertices emitting infinitely many
edges are to be dealt with appropriately. We first show that if a vertex $v$
in graph $E$ emits infinitely many edges, then the corresponding simple module
$\mathbf{S}_{v}$ defined in \cite{R} is not finitely presented. Thus the graph
$E$ must be row-finite if every simple left module over $L_{K}(E)$ is to be
finitely presented. Generalization of a Lemma from \cite{AR1} to the case of
row-finite graphs shows that simple modules induced by infinite irrational
paths not contaning a line point (see definition below) are also not finitely
presented. Eliminating these possibilities, we are able to obtain a complete
characterization (both graphical and algebraic) of Leavitt path algebras
$L_{K}(E)$ over which every simple left/right module is finitely presented,
thus leading to the following main theorem. Here we denote a pre-order $\geq$
among cycles by writing $c\geq c^{\prime}$ for two cycles $c$ and $c^{\prime}$
if there is a path connecting a vertex on $c$ to a vertex on $c^{\prime}$.

\begin{theorem}
Let $E$ be an arbitrary graph, $K$ be any field and let $L=L_{K}(E)$. Then the
following statements are equivalent:

(1) \ Every simple left/right $L$-module is finitely presented;

(2) \ $L$ is the union of a continuous well-ordered ascending chain of graded
ideals%
\[
0\leq I_{1}<\cdot\cdot\cdot<I_{\alpha}<I_{\alpha+1}<\cdot\cdot\cdot
\qquad\qquad\qquad(\alpha<\tau)\qquad\qquad(\ast\ast\ast)
\]
where $\tau$ is a suitable ordinal, $I_{1}=Soc(L)$ and, for each $\alpha\geq1$
with $I_{\alpha}\neq L$, $I_{\alpha+1}/I_{\alpha}\cong M_{\Lambda_{\alpha}%
}(K[x,x^{-1}])$, where $\Lambda_{\alpha}$ is an arbitrary index set (depending
on $\alpha$).\ 

(3) $E$ is row-finite, and either (a) $E^{0}$ is the saturated closure of the
set of all line points in $E$ (and is, in particular, acyclic) or (b)(i) $E$
contains cycles and the set $C$ of all the cycles in $E$ becomes an artinian
partially ordered set under the relation $\geq$, (ii) every infinite path in
$E$ either contains a line point or is tail equivalent to a rational path and
(iii) For every proper hereditary saturated subset of vertices $H$ containing
all the lines points in $E$, $E\backslash H$ contains cycles without exits but
does not contain any line points.
\end{theorem}

Observing that, for a finite graph $E$, Condition (3)(b)(i) of the above
theorem is equivalent to the condition that distinct cycles in $E$ have no
common vertex and that Conditions (3)(b)(ii) and (iii) are automatically
satisfied, we obtain the main theorem of \cite{AR1}:

\begin{corollary}
\cite{AR1} If $E$ is a finite graph, then every simple left module over
$L_{K}(E)$ is finitely presented if and only if distinct cycles in $E$ are
disjoint, that is, have no common vertex.
\end{corollary}

Interestingly, the graphical condition for a finite graph $E$ in the preceding
corollary (that distinct cycles in $E$ have no common vertex) has been shown
in \cite{AAJZ1} to be equivalent to the condition that the corresponding
Leavitt path algebra $L_{K}(E)$ has finite Gelfand-Kirillov dimension (for
short, GK-dimension). A natural question is whether this equivalence extends
to arbitrary graphs. After constructing examples showing that this equivalence
no longer holds for infinite graphs, use of a direct limit construction done
in \cite{AR} leads to an easy extension of the result of \cite{AAJZ1} to
arbitrary graphs (Theorem \ref{Finite GK-dim} ). These algebras seem to be
"made up" of von Neumann regular rings and the Laurent polynomial ring
$K[x,x^{-1}]$.

A (directed) graph $E=(E^{0},E^{1},r,s)$ consists of two sets $E^{0}$ and
$E^{1}$ together with maps $r,s:E^{1}\rightarrow E^{0}$. The elements of
$E^{0}$ are called \textit{vertices} and the elements of $E^{1}$
\textit{edges}. We generally follow the notation, terminology and results from
\cite{AAS}, \cite{AA} and \cite{AMP}. We outline some of the concepts and
results that will be used in this paper.

A vertex $v$ is called a \textit{sink} if it emits no edges, that is,
$s^{-1}(v)=\emptyset$, the empty set. The vertex $v$ is called a
\textit{regular vertex} if $s^{-1}(v)$ is finite and non-empty and $v$ is
called an \textit{infinite emitter} if $s^{-1}(v)$ is infinite. For each $e\in
E^{1}$, we call $e^{\ast}$ a ghost edge. We let $r(e^{\ast})$ denote $s(e)$,
and we let $s(e^{\ast})$ denote $r(e)$. A \textit{finite path} $\mu$ of length
$n>0$ is a finite sequence of edges $\mu=e_{1}e_{2}\cdot\cdot\cdot e_{n}$ with
$r(e_{i})=s(e_{i+1})$ for all $i=1,\cdot\cdot\cdot,n-1$. In this case
$\mu^{\ast}=e_{n}^{\ast}\cdot\cdot\cdot e_{2}^{\ast}e_{1}^{\ast}$ is the
corresponding ghost path. Any vertex $v$ is considered a path of length $0$.
The set of all vertices on a path $\mu$ is denoted by $\mu^{0}$.

Given an arbitrary graph $E$ and a field $K$, the \textit{Leavitt path algebra
}$L_{K}(E)$ is defined to be the $K$-algebra generated by a set $\{v:v\in
E^{0}\}$ of pairwise orthogonal idempotents together with a set of variables
$\{e,e^{\ast}:e\in E^{1}\}$ which satisfy the following conditions:

(1) \ $s(e)e=e=er(e)$ for all $e\in E^{1}$.

(2) $r(e)e^{\ast}=e^{\ast}=e^{\ast}s(e)$\ for all $e\in E^{1}$.

(3) (The "CK-1 relations") For all $e,f\in E^{1}$, $e^{\ast}e=r(e)$ and
$e^{\ast}f=0$ if $e\neq f$.

(4) (The "CK-2 relations") For every regular vertex $v\in E^{0}$,
\[
v=\sum_{e\in E^{1},s(e)=v}ee^{\ast}.
\]

A subgraph $F$ of a graph $E$ is called a \textit{complete subgraph }if, for
any vertex $v\in F$, $s_{F}^{-1}(v)=s_{E}^{-1}(v)$. In this case the
subalgebra generated by $F$ is isomorphic to $L_{K}(F)$.

A path $\mu$ $=e_{1}\dots e_{n}$ in $E$ is \textit{closed} if $r(e_{n}%
)=s(e_{1})$, in which case $\mu$ is said to be \textit{based at the vertex
}$s(e_{1})$. A closed path $\mu$ as above is called \textit{simple} provided
it does not pass through its base more than once, i.e., $s(e_{i})\neq
s(e_{1})$ for all $i=2,...,n$. The closed path $\mu$ is called a
\textit{cycle} if it does not pass through any of its vertices twice, that is,
if $s(e_{i})\neq s(e_{j})$ for every $i\neq j$. An \textit{exit }for a path
$\mu=e_{1}\dots e_{n}$ is an edge $f$ such that $s(f)=s(e_{i})$ for some $i$
and $f\neq e_{i}$. The graph $E$ is said to satisfy \textit{Condition (L)} if
every closed path has an exit. $E$ is said to satisfy \textit{Condition (K)}
if each vertex in $E$ is the base of either no closed path or at least two
distinct closed paths. Condition (K) always implies Condition (L).

A subset $H$ of $E^{0}$ is called \textit{hereditary} if, whenever $v\in H$
and there is a path from $v$ to $w\in E^{0}$, then $w\in H$. A hereditary set
is \textit{saturated} if, for any regular vertex $v$, $r(s^{-1}(v))\subseteq
H$ implies $v\in H$. If $E$ is row-finite and $I$ is the ideal generated by a
hereditary saturated set $H$ of vertices, then $L/I\cong L_{K}(E\backslash H)$
where $E\backslash H$ is the "quotient graph" defined by setting $(E\backslash
H)^{0}=E^{0}\backslash H$ and $(E\backslash H)^{1}=\{e\in E^{1}:r(e)\notin
H\}$, and the maps $r,s$ are the same (see, \cite{AMP}). Moreover, every
element of $I$ is a $K$-linear combination of monomials of the form $pq^{\ast
}$ where $r(p)=r(q)\in H$.

We shall also be using the following concepts and results from \cite{T}.

Let $E$ be an arbitrary graph and let $H$ be a hereditary saturated subset of
vertices in $E$. An infinite emitter $v$ is called a \textit{breaking vertex
}for $H$ if $0<|s^{-1}(v)\cap r^{-1}(E^{0}\backslash H)|<\infty$. The set of
all breaking vertices for $H$ is denoted by $B_{H}$. If $v\in B_{H}$, the
$v^{H}$ denotes the element $v-%
%TCIMACRO{\tsum \limits_{e\in s^{-1}(v),r(e)\neq H}}%
%BeginExpansion
{\textstyle\sum\limits_{e\in s^{-1}(v),r(e)\neq H}}
%EndExpansion
ee^{\ast}$. If $I$ a graded ideal of $L_{K}(E)$ with $I\cap E^{0}=H$ and
$S=\{v\in B_{H}:v^{H}\in I\}$, then it was shown in \cite{T} that $I$ is the
ideal generated by $H\cup\{v^{H}:v\in S\}$ and is denoted by $I(H,S)$. It was
also shown in \cite{T} that $L_{K}(E)/I(H,S)\cong L_{K}(E\backslash(H,S))$
\ where $E\backslash(H,S)$ is the quotient graph given by

$(E\backslash(H,S))^{0}:=E^{0}\backslash H\cup\{u^{\prime}:u\in B_{H}%
\backslash S\};$

$(E\backslash(H,S))^{1}:=\{e\in E^{1}:r(e)\notin H\}\cup\{e^{\prime}:e\in
E^{1},r(e)\in B_{H}\backslash S\}.$

Here $r$ and $s$ are extended to $(E\backslash(H,S))^{0}$ by setting
$s(e^{\prime})=s(e)$ and $r(e^{\prime})=r(e)^{\prime}$. We shall using the
fact that $u^{\prime}$ is a sink for each $u\in B_{H}\backslash S$.

Given an infinite path $p=e_{1}e_{2}\cdot\cdot\cdot e_{n}\cdot\cdot\cdot$
\ and an integer $n\geq1$, Chen (\cite{C}) defines $\tau_{\leq n}%
(p)=e_{1}\cdot\cdot\cdot e_{n}$ and $\tau_{>n}(p)=e_{n+1}e_{n+2}\cdot
\cdot\cdot$ . Two infinite paths $p,q$ are said to be \textit{tail-equivalent}
if there exist positive integers $m,n$ \ such that $\tau_{>m}(p)=\tau_{>n}%
(q)$. This is an equivalence relation and the equivalence class of all paths
tail equivalent to an infinite path $p$ is denoted by $[p]$. An infinite path
$p$ is called a \textit{rational path} if $p=ggg\cdot\cdot\cdot$ where $g$ is
some (finite) closed path in $E$. Given an infinite path $p$, Chen defines
$V_{[p]}=%
%TCIMACRO{\tbigoplus \limits_{q\in\lbrack p]}}%
%BeginExpansion
{\textstyle\bigoplus\limits_{q\in\lbrack p]}}
%EndExpansion
Kq$, a $K$-vector space having $\{q:q\in\lbrack p]\}$ as a basis. $V_{[p]}$ is
made a left $L$-module by defining the module operation $\cdot$, for all
$q\in\lbrack p]$ and all $v\in E^{0}$, $e\in E^{1}$, as follows:

1) $v\cdot q=q$ or $0$ according as $v=s(q)$ or not;

2) $e\cdot q=eq$ or $0$ according as $r(e)=s(q)$ or not;

3) $e^{\ast}\cdot q=\tau_{>1}(q)$ or $0$ according as $q=eq^{\prime}$ or not.

In \cite{C}, Chen shows that under the above action of $L$, $V_{[p]}$ becomes
a simple left $L$-module which we shall call a Chen simple module.

Following Chen, it was shown in \cite{R} that if a vertex $v$ is an infinite
emitter, then the $K$-vector space $\mathbf{S}_{v}$ having as a basis the set
$B=\{p:p$ a path in $E$ with $r(p)=v\}$ can be a made a simple $L_{K}%
(E)$-module where the multilication operation $\cdot$ on elements of $B$ by
elements of $L_{K}(E)$ is induced by conditions 1), 2), 3) above plus the
additional condition that $e^{\ast}\cdot v=0$ for all edges $e\in E^{1}$. In
particular, $\beta^{\ast}\cdot v=0$ in $\mathbf{S}_{v}$ for all paths $\beta$
in $E$.

For any vertex $v$ in $E$, the \textit{tree} of $v$ is denoted by $T_{E}(v)$
and is defined as $T_{E}(v)=\{w\in E^{0}:$ there is a path from $v$ to $w\}$.
We say there is a \textit{bifurcation }at a vertex $v$, if $v$ emits more than
one edge. In a graph $E$, a vertex $v$ is called a \textit{line point} if
there is no bifurcation or a cycle based at any vertex in $T_{E}(v)$. Thus, if
$v$ is a line point, there will be a single finite or infinite line segment
$\mu$ starting at $v$ ($\mu$ could just be $v$) and any other path $\alpha$
with $s(\alpha)=v$ will just be an initial sub-segment of $\mu$. It was shown
in \cite{AMMS} that $v$ is a line point in $E$ if and only if $vL_{K}(E)$ (and
likewise $L_{K}(E)v$) is a simple left (right) ideal. Moreover, the ideal
generated by all the line points in $E$ is the socle of $L_{K}(E)$. If $v$ is
a line point, then it is clear that any $w\in T_{E}(v)$ is also a line point.

Even though the Leavitt path algebra $L_{K}(E)$ may not have the
multiplicative identity $1$, we shall write $L_{K}(E)(1-v)$ to denote the set
$\{x-xv:x\in L_{K}(E)\}$. If $v$ is an idempotent (in particular, a vertex),
we then get a direct decomposition $L_{K}(E)=L_{K}(E)v\oplus L_{K}(E)(1-v)$.

Recall that a ring $R$ is von Neumann regular if for every element $a\in R$,
there is a $b\in R$ such that $a=aba$.

\section{When the graph $E$ contains no cycles}

In this section, we describe all the acyclic graphs $E$ such that every simple
left/right module over $L_{K}(E)\ $is finitely presented.

We begin with a useful Remark.

\begin{remark}
\label{Unitization}Let $E$ be an arbitrary graph and $K$ be any field. Let
$L=L_{K}(E)$ and let $L^{1}=L\times K$, be the unitization of $L$ where the
addition in $L^{1}$ is termwise and the multiplication is given by
$(a,r)(b,s)=(ab+rb+sa,rs)$. Identifying $L$ with the set $\{(a,0):a\in L\}$,
we notice that $L$ is an ideal of $L^{1}$ and that $L^{1}/L\cong K$. So $L$ is
von Neumann regular if and only if $L^{1}$ is.

Also if $M$ is any left $L$-module that is unital (i.e. $LM=M$), then $M$ is
also a left $L^{1}$-module. Because, for any $x\in M$, there is a local unit
$u\in L$ such that $ux=x$ and so, for any $r_{1}\in L^{1}$, we can define
$r_{1}x=r_{1}ux=(r_{1}u)x\in M$. From Proposition 2.4 of \cite{ARS}, every
projective left $L$-module is also a projective $L^{1}$-module.
\end{remark}

\begin{theorem}
\label{No cycles}Let $E$ be an arbitrary acyclic graph, $K$ be any field and
$L=L_{K}(E)$. Then the following are equivalent:

(i) \ \ Every simple left/right $L$-module is finitely presented;

(ii) $\ L=Soc(L)$ is a direct sum of simple left/right ideals;

(iii) $E^{0}$ is the saturated closure of the set of all line points in $E$.
\end{theorem}

\begin{proof}
Assume (i). We wish to show that every simple left $L$-module $S$ is
projective. Since $S$ is finitely presented, we can write $S=P/N$ where $P$ is
a projective $L$-module and both $P$ and $N$ are finitely generated. Now $E$
has no cycles. So $L$ (and hence, its unitization $L^{1}$) is von Neuman
regular by \cite{AR} and Remark \ref{Unitization}. Also, as noted above, $P$
is also a projective left module over $L^{1}$. On the other hand, it is known
(see Theorem 1.11, \cite{Good}) that every finitely generated submodule of the
projective $L^{1}$-module $P$ is a direct summand as a left $L^{1}$-module and
hence as a left $L$-module. This means that $N$ is a direct summand of $P$ and
hence $S$ is projective. From the conclusion that every simple left $L$-module
is projective, one can then easily show (applying Zorn's Lemma to the direct
sums of simple left ideals in $L$) that $L$ is a direct sum of simple left
ideals (see Proposition 2.27, \cite{ARS}). A similar argument and conclusion
holds for right $L$-modules. This proves (ii).

Now (ii) $\Leftrightarrow$ (iii) follows from the fact that $Soc(L)$ is the
ideal generated by the set of all line points in $E$ (see \cite{AMMS}) and
that $Soc(L)\cap E^{0}$ is the saturated closure of the hereditary set of all
line points in $E$.

Assume (ii). Since $L=Soc(L)$ is a semisimple left/right $L$-module, every
submodule of $L$ and, in particular, every maximal submodule of $L$ is a
direct summand of $L$. Since $L$ is itself projective (see \cite{ARS}), then
every simple left/right $L$-module is projective and so is finitely presented.
This proves (i).
\end{proof}

\textbf{Remark}: Since a line point is either a sink or a regular vertex, it
clear from the definition of the saturated closure, that Condition (iii) (that
$E^{0}$ is the saturated closure of the set of line points) implies that every
vertex $u\in E^{0}$ is a regular vertex. Hence the graph $E$ must be
row-finite. As we shall see later, a direct argument shows that the same
conclusion holds even when $E$ contains cycles. Also Condition (iii) of
Theorem \ref{No cycles} is equivalent to the statement that given any vertex
$v$ in $E$, there is an integer $n>0$ such that every path from $v$ of length
$>n$ ends at a line point (see Lemma 1.4, \cite{AbRS}).

\section{When the graph $E$ contains cycles \ }

We begin with an easy Lemma.

\begin{lemma}
\label{max ideal}Let $E$ be an arbitrary graph and $L=L_{K}(E)$. For every
maximal left ideal $M$ of $L$, there is exactly one vertex $u\notin M$ such
that $M=(M\cap Lu)\oplus%
%TCIMACRO{\tbigoplus \limits_{v\in E^{0},v\neq u}}%
%BeginExpansion
{\textstyle\bigoplus\limits_{v\in E^{0},v\neq u}}
%EndExpansion
Lv$. Thus every simple left $L$-module will be isomorphic to $Lu/N$ for some
vertex $u$ in $L$ and a maximal submodule $N$ of $Lu$. Similar statements hold
for a maximal right ideal of $L$.
\end{lemma}

\begin{proof}
Since $L=%
%TCIMACRO{\tbigoplus \limits_{v\in E^{0}}}%
%BeginExpansion
{\textstyle\bigoplus\limits_{v\in E^{0}}}
%EndExpansion
Lv$ and $M\neq L$, there is a $u\in E^{0}\backslash M$. Note that $M\cap
Lu\subset Mu$ and so $M\cap Lu=Mu$, by maximality. Then writing for any $m\in
M$, $m=mu+(m-mu)$, we get $M=Mu\oplus M(1-u)\subset Mu\oplus L(1-u)$. By
maximality, $M=Mu\oplus L(1-u)=(M\cap Lu)\oplus%
%TCIMACRO{\tbigoplus \limits_{v\in E^{0}\backslash\{u\}}}%
%BeginExpansion
{\textstyle\bigoplus\limits_{v\in E^{0}\backslash\{u\}}}
%EndExpansion
Lv$.
\end{proof}

For convenience sake, hereafter we shall consider only left $L$-modules. By
symmetry, all our results also hold for right $L$-modules.

If the simple module $S=Lu/N$ is, in addition finitely presented, then $S\cong
P/M^{\prime}$ where $P$ is a finitely generated projective $L$-module and
$M^{\prime}$ is finitely generated. In that case, by Schanuel's Lemma
\cite{La}, $N$ will also be finitely generated. So, for an arbitrary graph
$E$, checking whether all the simple left $L_{K}(E)$-modules are finitely
presented or not is equivalent to checking whether, for every vertex $u$,
every maximal submodule of $L_{K}(E)u$ is finitely generated or not.

It was shown in \cite{R} that every infinite emitter $v$ gives rise to a
simple $L$-module $(\mathbf{S}_{v},\cdot)$ which has as a $K$-basis the set of
all the paths in $E$ that end at $v$ and $\mathbf{S}_{v}$ has the $L$-module
operation $\cdot$ as indicated in the Preliminaries section. The next
proposition shows that, for any infinite emitter $v$, the simple module
$\mathbf{S}_{v}$ is not finitely presented.

\begin{proposition}
\label{Infinite emitter not fp}Let $E$ be an arbitrary graph. If $v\in E^{0}$
is an infinite emitter, then the corresponding simple left module
$(\mathbf{S}_{v},\cdot)$ over $L=L_{K}(E)$ is not finitely presented.
\end{proposition}

\begin{proof}
Suppose, on the contrary, $S_{v}$ is finitely presented. Writing
$\mathbf{S}_{v}=L\cdot v$, consider the exact sequence
\[
0\longrightarrow M\longrightarrow L\overset{\theta}{\longrightarrow}%
\mathbf{S}_{v}\longrightarrow0
\]
where $\theta(a)=a\cdot v$ for all $a\in L$. By Lemma \ref{max ideal}, $M=(%
%TCIMACRO{\tbigoplus \limits_{u\in E^{0},u\neq v}}%
%BeginExpansion
{\textstyle\bigoplus\limits_{u\in E^{0},u\neq v}}
%EndExpansion
Lu)\oplus N$ where $N=M\cap Lv$. Restricting $\theta$ to $Lv$ we get an exact
sequence%
\[
0\longrightarrow N\longrightarrow Lv\overset{\theta}{\longrightarrow
}\mathbf{S}_{v}\longrightarrow0,
\]
where $N$ is a finitely generated left ideal by our supposition. Let
$x_{1},\cdot\cdot\cdot,x_{k}$ be the generators of $N$. For each
$r=1,\cdot\cdot\cdot,k$, we can write $x_{r}=%
%TCIMACRO{\tsum \limits_{i=1}^{m_{r}}}%
%BeginExpansion
{\textstyle\sum\limits_{i=1}^{m_{r}}}
%EndExpansion
k_{i}\alpha_{i}\beta_{i}^{\ast}$ where $m_{r}$ is some positive integer and,
for all $i=1,\cdot\cdot\cdot,m_{r}$, $v=s(\beta_{i})$ and $r(\alpha
_{i})=r(\beta_{i})$. Now every non-zero term in the summation for $x_{r}$ must
involve a ghost path $\beta_{i}^{\ast}$. Because, otherwise, re-indexing the
terms, we can write $x_{r}=%
%TCIMACRO{\tsum \limits_{i=1}^{s}}%
%BeginExpansion
{\textstyle\sum\limits_{i=1}^{s}}
%EndExpansion
k_{i}\alpha_{i}+%
%TCIMACRO{\tsum \limits_{j=s+1}^{m_{r}}}%
%BeginExpansion
{\textstyle\sum\limits_{j=s+1}^{m_{r}}}
%EndExpansion
k_{j}\alpha_{j}\beta_{j}^{\ast}$ where, for all $i=1,\cdot\cdot\cdot,s$, we
assume that the real paths $\alpha_{i}$ are all different, that $k_{i}\neq0$
and that $r(\alpha_{i})=v$\ (as $x_{r}v=x_{r}$). Then, since $\beta_{j}^{\ast
}\cdot v=0$ for all $j$, we obtain $0=\theta(x_{r})=x_{r}\cdot v=(%
%TCIMACRO{\tsum \limits_{i=1}^{r}}%
%BeginExpansion
{\textstyle\sum\limits_{i=1}^{r}}
%EndExpansion
k_{i}\alpha_{i})\cdot v=%
%TCIMACRO{\tsum \limits_{i=1}^{r}}%
%BeginExpansion
{\textstyle\sum\limits_{i=1}^{r}}
%EndExpansion
k_{i}\alpha_{i}$ in $\mathbf{S}_{v}$, a contradiction as the paths $\alpha
_{i}$ are $K$-independent. Thus each $x_{r}$ is a $K$-linear combination of
finitely many monomials of the form $\alpha_{i}\beta_{i}^{\ast}$. So $N=%
%TCIMACRO{\tsum \limits_{i=1}^{n}}%
%BeginExpansion
{\textstyle\sum\limits_{i=1}^{n}}
%EndExpansion
L\alpha_{i}\beta_{i}^{\ast}$, where $n$ is some positive integer and for each
$i=1,\cdot\cdot\cdot,n$, $s(\beta_{i})=v$ and $r(\alpha_{i})=r(\beta_{i})$.
Since $v$ is an infinite emitter, we can choose an edge $f$ with $s(f)=v$
which is not the initial edge of any of the paths $\beta_{i}$ and so
$\beta_{i}^{\ast}f=0$ for $i=1,\cdot\cdot\cdot,n$. Now $ff^{\ast}\in N$ (as
$ff^{\ast}\cdot v=0$) and so we can write $ff^{\ast}=%
%TCIMACRO{\tsum \limits_{i=1}^{n}}%
%BeginExpansion
{\textstyle\sum\limits_{i=1}^{n}}
%EndExpansion
b_{i}\alpha_{i}\beta_{i}^{\ast}$ where $b_{i}\in L$. But then $ff^{\ast
}=ff^{\ast}ff^{\ast}=$ $%
%TCIMACRO{\tsum \limits_{i=1}^{n}}%
%BeginExpansion
{\textstyle\sum\limits_{i=1}^{n}}
%EndExpansion
b_{i}\alpha_{i}\beta_{i}^{\ast}ff^{\ast}=0$, a contradiction. Hence $N$ is not
finitely generated and consequently $\mathbf{S}_{v}$ is not finitely presented.
\end{proof}

\begin{corollary}
\label{row-finiteness}Let $E$ be an arbitrary graph. If every simple left
module over $L_{K}(E)$ is finitely presented, then $E$ must be a row-finite graph.
\end{corollary}

In view of Corollary \ref{row-finiteness}, we assume hereafter that $E$ is a
row-finite graph.

\begin{remark}
Let $p$ be an infinite path. If the corresponding Chen simple module $V_{[p]}$
is projective, then it is clearly finitely presented. In this case, Chen
\cite{C} showed that path $p$ will be tail equivalent to the infinite line
segment $\bullet\longrightarrow\bullet\longrightarrow\bullet\longrightarrow
\cdot\cdot\cdot$. Thus $V_{[p]}$ will be projective if and only if the path
$p$ contains a line point.
\end{remark}

It was shown in (Proposition 4.1, \cite{AR1}) that if $E$ is a finite graph,
then for any infinite irrational path $p$, the Chen simple module $V_{[p]}$ is
not finitely presented. We wish to generalize this result to the case when $E$
is a row-finite graph. Specifically we show that if $p$ is an infinite
irrational path in a row-finite graph $E$ such that no vertex on $p$ is a line
point, then the Chen simple module $V_{[p]}$ is not finitely presented. To
accomplish this, we fix some notation and make some initial useful remarks.

Let $E$ be a row-finite graph and let $p=e_{1}e_{2}\cdot\cdot\cdot e_{n}%
\cdot\cdot\cdot$ be an infinite irrational path in $E$ such that no vertex on
the path $p$ is a line point. This means that there is an infinite sequence of
integers $n_{1}<\cdot\cdot\cdot\cdot<n_{i}<n_{i+1}\cdot\cdot\cdot$ such that,
for each $i$, there is a bifurcation at $s(e_{n_{i}})=v_{n_{i}}$. For
convenience, we shall call the $n_{i}$ \textit{bifurcating integers}. Let
$v=s(e_{1})=s(p)$. Writing the simple module $V_{[p]}=L\cdot p$, we have an
exact sequence of $L$-modules%
\[
0\longrightarrow N\longrightarrow Lv\overset{\theta}{\longrightarrow}%
V_{[p]}\longrightarrow0\qquad\qquad\qquad\qquad(\ast)
\]
where $\theta(a)=a\cdot p$ for all $a\in Lv$. Observe that if $f$ is a
bifurcating edge with $s(f)=v_{n_{i}}$, then $\theta(f^{\ast}e_{n_{i}-1}%
^{\ast}\cdot\cdot\cdot e_{1}^{\ast})=0$ and so, for each $i$, the left ideal
\[
L_{n_{i}}=%
%TCIMACRO{\tsum \limits_{f\in s^{-1}(v_{n_{i}}),f\neq e_{n_{i}}}}%
%BeginExpansion
{\textstyle\sum\limits_{f\in s^{-1}(v_{n_{i}}),f\neq e_{n_{i}}}}
%EndExpansion
Lf^{\ast}e_{n_{i}-1}^{\ast}\cdot\cdot\cdot e_{1}^{\ast}\subseteq N.
\]
It is an easy argument to show that $%
%TCIMACRO{\tsum \limits_{i\geq1}}%
%BeginExpansion
{\textstyle\sum\limits_{i\geq1}}
%EndExpansion
L_{n_{i}}=%
%TCIMACRO{\tbigoplus \limits_{i\geq1}}%
%BeginExpansion
{\textstyle\bigoplus\limits_{i\geq1}}
%EndExpansion
L_{n_{i}}$. Indeed, suppose%
\[
a_{n_{1}}+\cdot\cdot\cdot+a_{n_{k}}=0\qquad\qquad\qquad(\#)
\]
where $a_{n_{i}}\in L_{n_{i}}$. Denoting $\mu_{n_{k}}=e_{1}\cdot\cdot\cdot
e_{n_{k}-1}e_{n_{k}-1}^{\ast}\cdot\cdot\cdot e_{1}^{\ast}$, observe that, for
any $i=1,\cdot\cdot\cdot\cdot,k$, $a_{n_{i}}\mu_{n_{k}}=a_{n_{k}}$ or $0$
according as $i=k$ or not. Then multiplying the equation $(\#)$ on the right
by $\mu_{n_{k}}$, we get $a_{n_{k}}=0$. Proceeding like this, we obtain that
$a_{n_{i}}=0$ for all $i=1,...,k$, thus showing that \ our sum is direct. So,
$%
%TCIMACRO{\tbigoplus \limits_{i\geq1}}%
%BeginExpansion
{\textstyle\bigoplus\limits_{i\geq1}}
%EndExpansion
L_{n_{i}}\subseteqq N$.

To establish that $V_{[p]}$ is not finitely presented, all we need is to show
that
\[
N=%
%TCIMACRO{\tbigoplus \limits_{i=1}^{\infty}}%
%BeginExpansion
{\textstyle\bigoplus\limits_{i=1}^{\infty}}
%EndExpansion
L_{n_{i}}.\qquad\qquad\qquad\qquad(\#\#)
\]
However, in \cite{AMT}, the authors investigate projective resolutions of
simple $L_{K}(E)$-modules and, in doing so, they proceed as we have done to
consider a simple module\ $V_{[p]}$ induced by infinite irrational path $p$
not containing a line point and indeed show, in our notation, that $N=%
%TCIMACRO{\tbigoplus \limits_{i=1}^{\infty}}%
%BeginExpansion
{\textstyle\bigoplus\limits_{i=1}^{\infty}}
%EndExpansion
L_{n_{i}}$ (see Lemmas 2.14 and 2.15 and Corollary 2.6 in \cite{AMT}). Since
our proof of $(\#\#)$ is essentially the same as given in \cite{AMT}, we omit
the proof and refer to \cite{AMT} for a justification of $(\#\#)$. A
consequence of $(\#\#)$ is the following.

\begin{corollary}
\label{Irrational => no fp}Let $E$ be a row-finite graph. If $p$ is an
infinite irrational path in $E$ containing no line points, then the Chen
simple module $V_{[p]}$ is \ not finitely presented.
\end{corollary}

Next, we recall a pre-order $\geq$ \ that was introduced in \cite{AR1} (see
also \cite{AAJZ1}) among the cycles in the graph $E$.

\begin{definition}
\label{Poset} (\cite{AR1}) Given two cycles $c,c^{\prime}$ in $E$, define
$c\geq c^{\prime}$ if there is a path from a vertex on $c$ to a vertex on
$c^{\prime}$.
\end{definition}

As a consequence of Corollary \ref{row-finiteness} and Corollary
\ref{Irrational => no fp}, we derive the following Proposition which
summarises the necessary conditions on the graph $E$ in order that every
simple left module over $L=L_{K}(E)$ is finitely generated.

\begin{proposition}
\label{disjoint cycles}Let $E$ be an arbitrary graph that contains cycles. If
every simple left $L$-module is finitely presented, then $E$ satisfies the
following properties:

(i) \ \ $E$ is row-finite;

(ii) \ \ Distinct cycles in $E$ have no common vertex;

(iii) \ The set $C$ of all cycles in $E$ is a non-empty artinian partially
ordered set under the relation $\geq$;

(iv) \ Every infinite path in $E$ either contains a line point or is tail
equivalent to a rational path.
\end{proposition}

\begin{proof}
(i) Follows from Corollary \ref{row-finiteness}.

(ii) If the same vertex $v$ is the base of two different cycles $g,h$, then we
get an infinite irrational path $p=ghg^{2}h^{2}\cdot\cdot\cdot g^{n}h^{n}%
\cdot\cdot\cdot$. Then Corollary \ref{Irrational => no fp} implies that the
simple module $V_{[p]}$ is not finitely presented, a contradiction.

(iiii) Now Condition (ii) implies that the relation $\geq$ is anti-symmetric.
Thus $\geq$ is a partial order. If there is an infinite descending chain of
distinct cycles in $C$, then this chain can be expanded to an infinite
irrational path $p$ in $E$. This leads \ to a contradiction since the
corresponding simple module $V_{[p]}$ is, by Corollary
\ref{Irrational => no fp}, not finitely presented. Thus $(C,\geq)$ is an
artinian partially ordered set.

(iv) Corollary \ref{Irrational => no fp} implies that if an infinite path in
$E$ is not rational, then it must contain a line point.
\end{proof}

In preparation for proving our main theorem, we recall the following
Definition of a "hedgehog" graph from \cite{AP}.

\begin{definition}
\label{Hedgehog}Suppose $E$ is a row-finite graph and $H$ is a non-empty
hereditary saturated subset of vertices in $E$.

Let $F(H)=\{$paths $\alpha=e_{1}\cdot\cdot\cdot e_{n}:n\geq1,r(e_{i})\notin H$
for $i=1,\cdot\cdot\cdot,n-1$ and $r(e_{n})\in H\}$.

Let $\bar{F}(H)=\{\bar{\alpha}:\alpha\in F(H)\}$.

Then the "hedgehog" graph $_{H}E=(_{H}E^{0},_{H}E^{1},s^{\prime},r^{\prime})$
is defined as follows:

(1) $_{H}E^{0}=H\cup F(H)$.

(2) $_{H}E^{1}=\{e\in E^{1}:s(e)\in H\}\cup\bar{F}(H)$.

(3) For every $e\in E^{1}$with $s(e)\in H$, $s^{\prime}(e)=s(e)$ and
$r^{\prime}(e)=r(e)$.

(4) For every $\bar{\alpha}\in\bar{F}(H)$, $s^{\prime}(\bar{\alpha})=\alpha$
and $r^{\prime}(\bar{\alpha})=r(\alpha)$.
\end{definition}

It was shown in (Lemma 1.2, \cite{AP}) that, for a row-finite graph $E$, if
$I$ is a graded ideal of $L_{K}(E)$ with $I\cap E^{0}=H$, then $I\cong
L_{K}(_{H}E)$. Thus, in particular, the graded ideal $I$ is a ring with local units.

Following \cite{AAPS}, we call a ring $R$ with local units
\textit{categorically left noetherian }if submodules of finitely generated
left $R$-modules are again finitely generated. It was shown in \cite{AAPS}
that, for any index set $\Lambda$, the matrix ring $M_{\Lambda}(K[x,x^{-1}])$
is categorically noetherian. Also, as a special case of Theorem 3.9 of
\cite{AAPS}, one obtains that, for a graph $E$, $L_{K}(E)\cong M_{\Lambda
}(K[x,x^{-1}])$ if and only if $E$ contains a unique cycle $c$ without exits,
$T_{E}(v)\cap c^{0}\neq\emptyset$ for every vertex $v$.and every infinite path
in $E$ is tail-equivalent to the infinite rational path $ccc\cdot\cdot\cdot$ .

The next Proposition plays a key role in proving our main theorem.

\begin{proposition}
\label{Extension by matrix rings} Suppose $E$ is a row-finite graph and $M$ is
a graded ideal of $L=L_{K}(E)$ such that $L/M\cong M_{\Lambda}K[x,x^{-1}]$
where $\Lambda$ is an arbitrary index set. If every simple left $M$-module is
finitely presented, then every simple left $L$-module is finitely presented.
\end{proposition}

\begin{proof}
Let $S$ be a simple left $L$-module.

Case 1: Suppose $MS=S$. Then $S$ is a simple $M$-module and, by hypothesis, is
finitely presented as an $M$-module. Let $H=M\cap E^{0}$. Since $M$ is a
graded ideal, $M\cong L_{K}(_{H}E)=L^{\prime}$ (\cite{AP}). By Lemma
\ref{max ideal}, $S\cong L^{\prime}u/A$ for some vertex $u\in(_{H}E)^{0}$
where $A$ is a finitely generated maximal $L^{\prime}$-submodule of
$L^{\prime}u$. Under the isomorphism $L^{\prime}\cong M$, let $u$ map to an
idempotent $\epsilon$ in $M$ and $A$ map to a submodule $B$ of $M$. Since $M$
has local units, $L\epsilon=M\epsilon$, $B$ is a maximal $L$-submodule of
$L\epsilon$ and $S\cong L\epsilon/B$. As $L$ is projective as a left
$L$-module (see \cite{ARS}), $L\epsilon$ is a cyclic projective summand of $L$
and $B$ is a finitely generated $L$-submodule. Hence $S\cong L\epsilon/B$ is
finitely presented as a left $L$-module.

Case 2: Suppose $MS=0$ so $S\cong L/Y$ for some maximal left ideal $Y\supseteq
M$. From Lemma \ref{max ideal}, it is clear that there is a vertex $v\notin Y$
such that $Y=(Lv\cap Y)\oplus L(1-v)$ and $S\cong Lv/N$, where $N=Lv\cap Y$.
If $v$ is a sink, then $Lv$ will be simple and a direct summand of $L$
(\cite{AMMS}) and so $S\cong Lv$ is projective and finitely presented. Suppose
$v$ is not a sink. Since $M$ is a two-sided ideal, $M=(Lv\cap M)\oplus
(L(1-v)\cap M)$ and clearly, $(Lv\cap M)\subset N=Lv\cap Y$. Let $H=M\cap
E^{0}$. Now $L_{K}(E\backslash H)\cong L/M\cong M_{\Lambda}(K[x,x^{-1}])$ for
some index set $\Lambda$. This means, by Theorem 3.9 of \cite{AAPS} (see also
Theorem 4.2.12, \cite{AAS}) that $E\backslash H$ has a unique cycle $c$
without exits based at a vertex $u$, $T_{E}(v)\cap c^{0}\neq\emptyset$ for
every vertex $v$ and every infinite path in $E\backslash H$ is tail-equivalent
to the the rational path $ccc\cdot\cdot\cdot$ . Now $(N+M)/M$ is a submodule
of the cyclic submodule $(Lv+M)/M$ and, since $L/M\cong M_{\Lambda}%
(K[x,x^{-1}])$ is categorically noetherian (see Lemma 1.3, \cite{AAPS}),
$N/(Lv\cap M)\cong(N+M)/M$ is finitely generated. Thus $N=Lx_{1}+\cdot
\cdot\cdot+Lx_{r}+(Lv\cap M)$ where the $x_{i}\in Lv$. So to prove that $S$ is
finitely presented, we need only to show that $Lv\cap M$ is finitely
generated. To start with, we claim that $T_{E\backslash H}(v)$ is a finite
set. To justify this, we follow an argument that is embedded in the proof of
Proposition 3.6 of \cite{AAPS}. Suppose, on the contrary, $T_{E\backslash
H}(v)$ is an infinite set. Then $v\notin c^{0}$, as $c^{0}$ is a hereditary
set due to $c$ having no exits and is finite. Since $v$ is a regular vertex,
there is then an edge $e_{1}$ with $s(e_{1})=v$, $r(e_{1})=v_{1}$ such that
$T_{E\backslash H}(v_{1})$ is an infinite set. Clearly, $v_{1}\notin c^{0}$
and $v_{1}$ is a regular vertex. Repeating this process, we obtain an infinite
path in $E\backslash H$ such that no vertex on this path lies on $c$. This
contradicts the fact the every infinite path in $E\backslash H$ is
tail-equivalent to the rational path $ccc\cdot\cdot\cdot$ . Thus
$T_{E\backslash H}(v)$ is a finite set.We now distinguish two cases.

Case A: Suppose $v\notin c^{0}$. Since $E\backslash H$ is row-finite, since
every infinite path in $E\backslash H$ is tail-equivalent to the rational path
$ccc\cdot\cdot\cdot$ and since $T_{E\backslash H}(v)$ is a finite set, the
number of paths $\alpha$ in $E\backslash H$ satisfying $s(\alpha)=v$ and
$r(\alpha)\notin c^{0}$ is finite. Among these finitely many paths, let
$\gamma_{1},\cdot\cdot\cdot,\gamma_{m}$ be the listing of all those paths with
the property that $s(\gamma_{j})=v$ and $r(\gamma_{j})=u_{j}$ such that there
is at least one $e\in s^{-1}(u_{j})$ with $r(e)\in H$. \ Here we use the
convention that $u_{j}=v$ if $\gamma_{j}$ has length $0$. For each
$j=1,\cdot\cdot\cdot,m$, let $\{e_{jk:}:k=1,\cdot\cdot\cdot,l_{j}\}$ be the
set of all the edges $e_{jk}\in s^{-1}(u_{j})$ such that $r(e_{jk})\in H$. Now
each element of $Lv\cap M$ is a $K$-linear combination of monomials of the
form $pq^{\ast}$ where $s(q)=v$ and $r(q)=r(p)\in H$. It is then clear that,
each such path $q$ is of the form $q=\gamma_{j}e_{jk}q^{\prime}$ for some $j$
and $k$, where $q^{\prime}$ is a suitable path. This means that $pq^{\ast}\in$
$Le_{jk}^{\ast}\gamma_{j}^{\ast}$. As $e_{jk}^{\ast}\gamma_{j}^{\ast}\in
Lv\cap M$, for all $j,k$, we then conclude that
\[
Lv\cap M=%
%TCIMACRO{\tsum \limits_{j=1}^{m}}%
%BeginExpansion
{\textstyle\sum\limits_{j=1}^{m}}
%EndExpansion%
%TCIMACRO{\tsum \limits_{k=1}^{l_{j}}}%
%BeginExpansion
{\textstyle\sum\limits_{k=1}^{l_{j}}}
%EndExpansion
Le_{jk}^{\ast}\gamma_{j}^{\ast}.
\]
Consequently, $N=%
%TCIMACRO{\tsum \limits_{i=1}^{r}}%
%BeginExpansion
{\textstyle\sum\limits_{i=1}^{r}}
%EndExpansion
Lx_{i}+(Lv\cap M)$ is finitely generated. This shows that the simple module
$S=Lv/N$ is finitely presented.

Case B: Suppose $v\in c^{0}$. In this case $Lv\cap M=\{0\}$, as there are no
paths connecting $v$ to a vertex in $H$. Then $N\cong N/(Lv\cap M)$ is
finitely generated and so $Lv/N$ is finitely presented.
\end{proof}

\bigskip

We are now ready to prove the main theorem of this section.

\begin{theorem}
\label{Non Acyclic case}Suppose $E$ is an arbitrary graph that contains cycles
and $L=L_{K}(E)$. Then the following conditions are equivalent:

(1) \ \ Every simple left $L$-module is finitely presented;

(2) \ (a) $E$ is row-finite, (b) the set $C$ of all the cycles in $E$ is a
non-empty artinian partially ordered set under the defined relation $\geq$,
(c) every infinite path in $E$ either contains a line point or is tail
equivalent to a rational path and (d) for every proper hereditary saturated
set $H$ of vertices containing all the line points, $E\backslash H$ contains
cycles without exits but no line points;

(3) $L$ is the union of a smooth well-ordered ascending chain of graded ideals%
\[
0\leq I_{1}<\cdot\cdot\cdot<I_{\alpha}<I_{\alpha+1}<\cdot\cdot\cdot
\qquad\qquad\qquad(\alpha<\tau)\qquad\qquad(\ast\ast\ast)
\]
where $\tau$ is a suitable ordinal, $I_{1}=Soc(L)$ (which may be $0$) and, for
each $\alpha\geq1$, $I_{\alpha+1}/I_{\alpha}\cong M_{\Lambda_{\alpha}%
}(K[x,x^{-1}])$ where $\Lambda_{\alpha}$ is an index set that depends on
$\alpha$.
\end{theorem}

\begin{proof}
Assume (1). We need only to prove Condition 2(d), as Conditions 2(a) - 2(c)
follow from Proposition \ref{disjoint cycles}. Let $H\neq E^{0}$ be a
hereditary saturated subset of vertices containing all the line points in $E$.
First we show that $E\backslash H$ contains no line points. Suppose, by way of
contradiction, $v$ is a line point in $(E\backslash H)^{0}=E^{0}\backslash H$.
Since $E$ is row-finite and $H$ contains all the sinks in $E$, $v$ is not a
sink in $E^{0}\backslash H$. So $T_{E\backslash H}(v)$ consists of the
infinite set $\{v=v_{1},v_{2},\cdot\cdot\cdot,v_{n},\cdot\cdot\cdot\}$ of
vertices having no bifurcations in $E\backslash H$ and they define an infinite
path%
\[
\underset{v_{1}}{\bullet}\longrightarrow\underset{v_{2}}{\bullet}\cdot
\cdot\cdot\cdot\underset{v_{n}}{\bullet}\longrightarrow\cdot\cdot\cdot\text{
.}%
\]
First note that, in $E$, none of these $v_{i}$ can be a base of a cycle $c$,
because otherwise $v_{i}$ will be a base of $c$ in $E\backslash H$, a
contradiction. Also none of the $v_{i}$ can be a line point in $E$ as $H$
contains all the line points. Thus these vertices $v_{i}$ define an infinite
irrational path in $E$ not containing any line points. This is impossible in
view of Corollary \ref{Irrational => no fp}. Hence $E\backslash H$ contains no
line points. As every simple left module over $L_{K}(E\backslash H)\cong
L/I(H)$ is finitely presented, Theorem \ref{No cycles} (iii) then implies that
$E\backslash H$ must contain a cycle. Moreover, by Proposition
\ref{disjoint cycles}, the cycles in $E\backslash H$ form a non-empty artinian
partially ordered set $C$ under the relation $\geq$. Since $E\backslash H$
does not contain any line points, Proposition \ref{disjoint cycles}(iv)
implies that any minimal element in $C$ will be a cycle without exits.\ 

Assume (2). Let $I_{1}$ be the ideal generated by all the line points in $E$.
Then $I_{1}$ is a graded ideal and is the socle of $L$ \cite{AMMS}. Note that
$I_{1}$ may be zero, but $I_{1}\neq L$ since $E$ contains cycles. Suppose
$\alpha>1$ and that the graded ideals $I_{\beta}$ have been defined for all
$\beta<\alpha$ with the stated properties. If $\alpha$ is a limit ordinal,
define $I_{\alpha}=\cup_{\beta<\alpha}I_{\beta}$. Suppose $\alpha=\beta+1$ and
that $I_{\beta}\neq L$. Let $H=I_{\beta}\cap E^{0}$. Now $E\backslash H$
satisfies the Conditions (2)(a) - (d) and, in particular, $E\backslash H$
contains cycles without exits. Since $L_{K}(E\backslash H)\cong L/I_{\beta}$,
identifying $L/I_{\beta}$ with $L_{K}(E\backslash H)$, define $I_{\beta
+1}/I_{\beta}$ to be the (graded) ideal generated by the vertices in a single
cycle without exits in $E\backslash H$. By (Proposition 3.7(iii),
\cite{AAPS}), $I_{\beta+1}/I_{\beta}$ is isomorphic to a matrix ring of the
form $M_{\Lambda_{\beta}}(K[x,x^{-1}])$ where $\Lambda_{\beta}$ is an
arbitrary index set. Proceeding like this and applying transfinite induction,
we obtain the transfinite chain $(\ast\ast\ast)$ of graded ideals, where the
successive quotients are matrix rings of appropriate size over $K[x,x^{-1}]$.
This proves (3).

Assume (3). We are given $L$ is the union of the transfinite chain $(\ast
\ast\ast)$ of graded ideals $I_{\alpha}$ with the stated properties. First of
all observe that, by Lemma \ref{max ideal}, every simple left $L$-module is
isomorphic to $Lv/N$ for some vertex $v$ in $E$. Now the vertex $v$ in $E$
belongs to some graded ideal $I_{\alpha}$ and each $I_{\alpha}$ is a ring with
local units as $I_{\alpha}\cong L_{k}(_{H_{\alpha}}E)$ where $H_{\alpha
}=I_{\alpha}\cap E^{0}$. This means that the $L$-submodules of $I_{\alpha}$
coincide with the $I_{\alpha}$-submodules of $I_{\alpha}$. \ Consequently,
every simple left $L$-module is a simple left $I_{\alpha}$-module for suitable
$\alpha$. So we wish to show, by\ transfinite induction on $\alpha$, that
every simple left $I_{\alpha}$-module is finitely presented as a simple left
$L$-module. If $\alpha=1$, this is immediate since $I_{1}$, being the socle
$L$, is a direct sum of projective simple left ideals of $L$.

Assume $\alpha\geq2$ and that, for all $\beta<\alpha$, every simple left
$I_{\beta}$-module is a finitely presented simple left $L$-module. Suppose
$\alpha$ is a limit ordinal so that $I_{\alpha}=%
%TCIMACRO{\tbigcup \limits_{\beta<\alpha}}%
%BeginExpansion
{\textstyle\bigcup\limits_{\beta<\alpha}}
%EndExpansion
I_{\beta}$. Since $I_{\alpha}$ is a graded ideal, $I_{\alpha}\cong
L_{K}(_{H_{\alpha}}E)$ where $H_{\alpha}=I_{\alpha}\cap E^{0}$. Now\ any
simple left $I_{\alpha}$-module $S$ is of the form $I_{\alpha}v/N$ where
$v\in(_{H_{\alpha}}E)^{0}$. Note that $v\in H_{\alpha}$ or $v\in F(H_{\alpha
})$ (see Defintion \ref{Hedgehog}). In either case, since $v\in I_{\beta}$ for
some $\beta<\alpha$, $S$ becomes a simple left $I_{\beta}$-module and so, by
induction, is a finitely presented simple left $L$-module. Suppose
$\alpha=\beta+1$ for some $\beta\geq1$. As before the graded ideal $I_{\alpha
}\cong L_{K}(_{H_{\alpha}}E)$ where $H_{\alpha}=I_{\alpha}\cap E^{0}$. Let $S$
be a simple left $I_{\alpha}$-module. Since $I_{\beta+1}/I_{\beta}\cong
M_{\Lambda_{\alpha}}(K[x,x^{1}])$ for some index set $\Lambda_{\alpha}$ and
since every simple $I_{\beta}$-module is finitely presented as an $I_{\beta}%
$-module (also as an $L$-module), we appeal to Proposition
\ref{Extension by matrix rings} to conclude that $S$ is finitely presented as
a simple left $I_{\alpha}$-module and hence is a finitely presented simple
left $L$-module. Applying transfinite induction, we reach the desired
conclusion. This proves (1).
\end{proof}

When $E$ is a finite graph, Conditions 2(a) and 2(c) of the above theorem are
immediate. Condition 2(d) is also automatically satisfied. To see this,
suppose $H$ is a proper hereditary saturated subset of vertices in the finite
graph $E$ containing all the line points. Since every vertex in $E$ is
regular, $E\backslash H$ contains no sinks (and no line points). This means
(since $E$ is finite) that every path in $E\backslash H$ eventually ends at a
cycle. In particular, Condition 2(c) holds. Also it is easy to see that the
condition that $\geq$ is antisymmetric (as part of Condition 2(b)) is
equivalent to stating that different cycles in $E$ have no common vertex. Thus
we derive the following characterization, proved in \cite{AR1}, of a Leavitt
path algebra $L$ of a finite graph $E$ over which every simple left $L$-module
is finitely presented.

\begin{corollary}
\label{Finite case}(\cite{AR1}) Let $E$ be any finite graph, $K$ be any field
and let $L=L_{K}(E)$. Then every simple left $L$-module is finitely presented
if and only if every vertex in $E$ is the base of at most one cycle.
\end{corollary}

EXAMPLE: As an example of a graph satisfying the conditions of Theorem
\ref{Non Acyclic case}, consider a graph $E^{\prime}$ consisting of infinitely
many loops $c_{i}$ based at vertices $v_{i}$ for $i=1,2,\cdot\cdot\cdot$ suh
that, for each $i$, there is an edge $e_{i}$ with $s(e_{i})=v_{i+1}$ and
$r(e_{i})=v_{i}$. In addition, there is a vertex $w$ and an edge $e$ with
$s(e)=v_{1}$ and $r(e)=w$. Thus $w$ is a sink and the only line point in $E$.
\ Let $H_{0}=\{w\}$ and, for each $n\geq1$, let $H_{n}=\{w,v_{1},\cdot
\cdot\cdot,v_{n}\}$. Clearly, the proper non-empty hereditary saturated
subsets of $(E^{\prime})^{0}$ are just the sets $H_{n}$, $n\geq0$. For each
$n\geq0$, the quotient graph $E^{\prime}\backslash H_{n}$ contains a cycle
without exits and has no line points. Moreover the cycles $c_{i}$ in
$E^{\prime}$ form an artinian partially ordered set under the relation $\geq$.

\bigskip

It is now clear that Theorem 1.1 follows from Theorems \ref{No cycles} and
\ref{Non Acyclic case}.

\section{A Corner-Tree Phenomena}

This section contains some preliminary results which will be used in the next
section. We explore the conditions needed for the corner $vLv$, where $v$ is
vertex, of a Leavitt path algebra $L=L_{K}(E)$ to have a various ring
properties. These seem to be governed by appropriate graph properties of the
tree $T_{E}(v)$. \ Our focus is when $v$ is an acyclic vertex (see definition
below) .

In the following, we make the convention that if $u,w\in T_{E}(v)$, then
$T_{E}(v)$ contains all the edges in the paths connecting $u$ to $w$. Thus
$T_{E}(v)$ is a complete subgraph of $E$.

We shall be using the following generalization of the "hedgehog" graph given
in Definition \ref{Hedgehog} to arbitrary graphs (see \cite{RT}).

\textbf{Result (a).} Let $E$ be an arbitrary graph and let $I$ a graded ideal
of $L_{K}(E)$ with $I\cap E^{0}=H$ and $S=\{v\in B_{H}:v^{H}\in I\}$. Then
Theorem 6.1 of \cite{RT} shows that $I=I(H,S)$ is isomorphic to a Leavitt path
algebra $L_{K}(\bar{E}(H,S))$ where the graph $\bar{E}(H,S)$ is defined as follows:

Let $F_{1}=\{$paths $\alpha=e_{1}\cdot\cdot\cdot e_{n}:n\geq1,r(e_{i})\notin
H$ for $i=1,\cdot\cdot\cdot,n-1$ and $r(e_{n})\in H\}$.

Let $F_{2}=\{$paths $\alpha:r(\alpha)\in S$ and length of $\alpha\geq1\}.$

For $i=1,2$, let $\bar{F}_{i}=\{\bar{\alpha}:\alpha\in F_{i}\}.$

Then $(\bar{E}(H,S))^{0}:=H\cup S\cup F_{1}\cup F_{2};$

$(\bar{E}(H,S))^{1}:=\{e\in E^{1}:s(e)\in H\}\cup\{e\in E^{1}:s(e)\in
S,r(e)\in H\}\cup\bar{F}_{1}\cup\bar{F}_{2},$

and $s,r$ are extended to $\bar{E}(H,S)$ by setting $s(\bar{\alpha})=\alpha$
and $r(\bar{\alpha})=r(\alpha)$ for all $\alpha\in F_{1}\cup F_{2}$.

Thus every graded ideal of $L_{K}(E)$, being isomorphic to a Leavitt path
algebra, is possessed with local units.

We begin with the following result from \cite{CO}.

\begin{lemma}
\label{Colak}(\cite{CO}) Let $\bar{H}$ be the saturated closure of a
hereditary set $H$ of vertices. If a vertex $w\in\bar{H}$ and $w$ lies on a
closed path, then $w\in H$.
\end{lemma}

As a consequence of the above Lemma, we obtain the following.

\begin{lemma}
\label{Cycles}Let $E$ be an arbitrary graph and $v\in E^{0}$. If the ideal $I$
generated by $v$ contains a closed path $c$ then $c^{0}\subset T_{E}(v)$.
\end{lemma}

\begin{proof}
This follows immediately if one observes that $T_{E}(v)$ is a hereditary set
and its saturated closure is $I\cap E^{0}$.
\end{proof}

It is known (see \cite{AMMS}) that a vertex $v$ is a line point (i.e.,
$T_{E}(v)$ is a single straight line segment) if and only if the corner
$vLv\cong K$. Likewise, by examining the representation of elements in $vLv$,
it is clear\ that $T_{E}(v)$ is a cycle without exits if and only if the
corner $vLv\cong K[x,x^{-1}]$. We explore below similar connections between
$T_{E}(v)$ and $vLv$.

\begin{proposition}
\label{acyclic vertex}Let $E$ be an arbitrary graph. Then the following are
equivalent for any vertex $v$ in $E$:

(i) The corner $vLv$ is von Neumann regular;

(ii) $T_{E}(v)$ contains no cycles;

(iii) The ideal $I$ generated by $v$ is von Neumann regular.
\end{proposition}

\begin{proof}
(i) =%
%TCIMACRO{\TEXTsymbol{>} }%
%BeginExpansion
$>$
%EndExpansion
(ii). Assume $vLv$ is von Neumann regular. Suppose, by way of contradiction,
$T_{E}(v)$ contains a cycle $c$. Let $p$ be a path from $v$ to a vertex $w$ on
$c$. Let $\gamma$ denote $v+pcp^{\ast}\in vLv$. We wish to show that there is
no $a=vav\in vLv$ such that $\gamma a\gamma=\gamma$. Suppose, by way of
contradiction, such an $a$ exists. Write $a=%
%TCIMACRO{\tsum \limits_{i=r}^{s}}%
%BeginExpansion
{\textstyle\sum\limits_{i=r}^{s}}
%EndExpansion
a_{i}$ as a graded sum of homogeneous elements where $r,s\in%
%TCIMACRO{\U{2124} }%
%BeginExpansion
\mathbb{Z}
%EndExpansion
$ with $r\leq s$. Substituting for $\gamma$ and $a$ in the equation $\gamma
a\gamma=\gamma$, we obtain the equation
\[
(v+pcp^{\ast})%
%TCIMACRO{\tsum \limits_{i=r}^{s}}%
%BeginExpansion
{\textstyle\sum\limits_{i=r}^{s}}
%EndExpansion
a_{i}(v+pcp^{\ast})=(v+pcp^{\ast}).
\]
To reach a contradiction, we essentially follow the ideas in the proof of
$(1)\Rightarrow(2)$ in Theorem 1 of \cite{AR} by expanding the above equation,
considering it as a graded equation and comparing the degree of the components
on both sides. Existence of terms of higher degrees on the left hand side with
no terms of equal degree on the right hand side leads to a contradiction.
Hence $T_{E}(v)$ contains no cycles, thus proving (ii).

(ii) =%
%TCIMACRO{\TEXTsymbol{>} }%
%BeginExpansion
$>$
%EndExpansion
(iii). If $T_{E}(v)$ contains no cycles, then by Lemma \ref{Cycles} $H=I\cap
E^{0}$ contains no cycles. Now, by Result (a), the graded ideal $I\cong
L_{K}(\bar{E}(H,\emptyset))$. From its definition it is clear that the graph
$\bar{E}(H,\emptyset)$ also contains no cycles. Then we appeal to Theorem 1 of
\cite{AR} to conclude that $I\cong L_{K}(\bar{E}(H,\emptyset))$ is von Neumann
regular, thus proving (iii).

(ii) =%
%TCIMACRO{\TEXTsymbol{>} }%
%BeginExpansion
$>$
%EndExpansion
(iii). If $I$ is von Neumann regular, then so is the corner $vLv=vIv$.
\end{proof}

For the convenience of later use, we introduce the following definition.

\begin{definition}
A vertex $v$ in a graph $E$ is called an \textbf{acyclic vertex} if it
satisfies the equivalent conditions of Proposition \ref{acyclic vertex}.
\end{definition}

As an immediate consequence of Proposition \ref{acyclic vertex}, we get the
following result.

\begin{corollary}
\label{acyclic vertex implies von neumann regular}Let $E$ be an arbitrary
graph and let $A$ be the ideal generated by the set of all the acyclic
vertices in $E$. Then $A$ is von Neumann regular.
\end{corollary}

REMARK: \textbf{As an aside}, we wish to point out that Proposition
\ref{acyclic vertex} appears to be one instance of the phenomenon that when a
graph property P of the graph $E$ is equivalent to a ring property Q of the
ring $L=L_{K}(E)$, then for any vertex $v\in E$, $T_{E}(v)$ has the property P
$\Longleftrightarrow$ $vLv$ has the property Q.

We digress a bit to point two other instances of such a phenomenaon. Two
special graph properties of $E$ which play an important role in the
investigation of $L_{K}(E)$ are Condition (L) and Condition (K). It was shown
in \cite{R1} that $E$ satisfies Condition (L) if and only if the Leavitt path
algebra $L=L_{K}(E)$ is a Zorn ring. Here, a ring $R$ is a \textit{Zorn ring}
if given any $a\in R$, there is a $b\in R$ such that $bab=b$ (For other
equivalent definitions, see \cite{R1}). Likewise, it was shown in \cite{ARS}
that the graph $E$ satisfies Condition (K) if and only if $L_{K}(E)$ is
right/left weakly regular. Recall that a ring $R$ is right \textit{weakly
regular} if to each $a\in R$ there is an $x\in RaR$ such that $a=ax$. (see
\cite{ARS} for other equivalent definitions and properties of left/right
weakly regular rings). In general, weak regularity and being a Zorn ring are
independent properties, neither implying the other (see \cite{R1} for examples).

\begin{proposition}
\label{Zorn rings}Let $E$ be an arbitrary graph and $L=L_{K}(E)$. Then, for a
vertex $v\in E$, the following are equivalent:

(i) \ The corner $vLv$ is a Zorn ring;

(ii) $T_{E}(v)$ satisfies Condition (L);

(iii) The ideal $I$ generated by $v$ is a Zorn ring.
\end{proposition}

\begin{proof}
(i) =%
%TCIMACRO{\TEXTsymbol{>} }%
%BeginExpansion
$>$
%EndExpansion
(ii). Assume $vLv$ is a Zorn ring. Suppose, on the contrary, $T_{E}(v)$
contains a cycle $c$ without exits in $T_{E}(v)$ and hence in $E$. Let $p$ be
a path from $v$ to a vertex $w$ on $c$. Then the map $\theta
:wLw\longrightarrow vLv$ given by $\theta(x)=pxp^{\ast}$ is a monomorphism.
Now $wLw\cong K[x,x^{-1}]$ as $c$ has no exits. So, for $\theta(w)=\epsilon$
we have $\epsilon L\epsilon\cong wLw\cong K[x,x^{-1}]$ and this is a
contradiction since, being a corner of the Zorn ring $vLv$, $\epsilon
L\epsilon=\epsilon vLv\epsilon$ must be a Zorn ring ( see \cite{R1}), but the
integral domain $K[x,x^{-1}]$ is obviously not a Zorn ring. Hence every cycle
in $T_{E}(v)$ must have an exit, thus proving (ii).

(ii) =%
%TCIMACRO{\TEXTsymbol{>} }%
%BeginExpansion
$>$
%EndExpansion
(iii). The proof is similar to that of (i) =%
%TCIMACRO{\TEXTsymbol{>} }%
%BeginExpansion
$>$
%EndExpansion
(ii) \ of Proposition \ref{acyclic vertex}. By Lemma \ref{Cycles}, $T_{E}(v)$
satisfies Condition (L) exactly when Condition (L) holds in $H=I\cap E^{0}$.
Then the graph $\bar{E}(H,\emptyset)$ also satisfies Condition (L).
Consequently, by Theorem 2.1 of \cite{R1}, $I\cong L_{K}(\bar{E}%
(H,\emptyset))$ is a Zorn ring.

(ii) =%
%TCIMACRO{\TEXTsymbol{>} }%
%BeginExpansion
$>$
%EndExpansion
(iii). If $I$ is a Zorn ring, then so is the corner $vLv=vIv$ (see \cite{R1}).
\end{proof}

\begin{proposition}
\label{weakly regular}Let $E$ be an arbitrary graph. Then the following are
equivalent for any vertex $v$ in $E$:

(i) $\ \ $The corner vLv is left/right weakly regular and a Zorn ring;

(ii) \ $T_{E}(v)$ satisfies Condition (K);

(iii) The ideal $I$ generated by $v$ is both left/right weakly regular and a
Zorn ring.
\end{proposition}

\begin{proof}
(i) =%
%TCIMACRO{\TEXTsymbol{>} }%
%BeginExpansion
$>$
%EndExpansion
(ii). Suppose $vLv$ is both left/right weakly regular and a Zorn ring. We wish
to show $T_{E}(v)$ satisfies Condition (K). Already by Proposition
\ref{Zorn rings}, $T_{E}(v)$ satisfies Condition (L). The proof that Condition
(K) holds follows the ideas in some earlier papers (see, for e.g., Proposition
3.9 in \cite{ARS}). Suppose Condition (K) does not hold in $T_{E}(v)$. Then
there is a vertex $w\in T_{E}(v)$ which is the base of only one closed path
(cycle) $c$. Now $c$ has exits in $T_{E}(v)$ due to Condition (L). Let $X$
denote the hereditary closure of the set $\{r(e):e$ an exit for $c\}$ and $Y$
be the saturated closure of $X$. Since $T_{E}(v)$ does not satisfy Condition
(K), Lemma \ref{Colak} implies that $Y\cap c^{0}=\emptyset$. If
$J=I(Y,\emptyset)$ is the (graded) ideal generated by $Y$, then $L/J\cong
L_{K}(E\backslash(Y,\emptyset))$ and in $E\backslash(Y,\emptyset)$, $c$ is a
cycle without exits based at $w$. For convenience, denote $L_{K}%
(E\backslash(Y,\emptyset))$ by $\bar{L}$. Let $p$ be a path connecting $v$ to
$w$ in $E\backslash(Y,\emptyset)$. Then, as was done in the proof of
Proposition \ref{Zorn rings}, the map $\theta:w\bar{L}w\longrightarrow
v\bar{L}v$ given by $\theta(x)=pxp^{\ast}$ is a monomorphism. Now $w\bar
{L}w\cong K[x,x^{-1}]$ as $c$ has no exits. If $\epsilon=\theta(w)$, then
$K[x,x^{-1}]\cong\epsilon\bar{L}\epsilon=\epsilon v\bar{L}v\epsilon$, a corner
of $v\bar{L}v$. This is a contradiction, since $v\bar{L}v$, being a
homomorphic image of $vLv$, is left/right weakly regular while the commutative
integral domain $K[x,x^{-1}]$ is not weakly regular. Hence $T_{E}(v)$
satisfies Condition (K), thus proving (ii).

(ii) =%
%TCIMACRO{\TEXTsymbol{>} }%
%BeginExpansion
$>$
%EndExpansion
(iii) Again the proof is similar to the proof of (i)=%
%TCIMACRO{\TEXTsymbol{>}}%
%BeginExpansion
$>$%
%EndExpansion
(ii) in Proposition \ref{acyclic vertex}. If $T_{E}(v)$ satisfies Condition
(K), then by Lemma \ref{Cycles}, Condition (K) holds in $H=I\cap E^{0}$ and
also in the graph $\bar{E}(H,\emptyset)$. Then, by \cite{ARS}, $I\cong
L_{K}(\bar{E}(H,\emptyset))$ is left/right weakly regular. Since Condition (K)
implies Condition (L), we appeal to Theorem 2.1 of \cite{R1} to conclude that
$L_{K}(\bar{E}(H,\emptyset))$ is also a Zorn ring. This proves (ii).

(iii) =%
%TCIMACRO{\TEXTsymbol{>} }%
%BeginExpansion
$>$
%EndExpansion
(i). This immediate from the fact that a corner of a left/right weakly regular
Zorn ring is again left/right weakly regular Zorn ring.
\end{proof}

\section{Leavitt path algebras with finite Gelfand-Kirillov dimension}

As we noted Corollary \ref{Finite case}, when $E$ is a finite graph all the
irreducible representations of $L_{K}(E)$ are finitely presented exactly when
distinct cycles in $E$ are disjoint. Interestingly, it was shown in
\cite{AAJZ1} that this same condition for a finite graph $E$ is equivalent to
the Leavitt path algebra $L_{K}(E)$ having finite Gelfand-Kirillov dimension.
Examples show that this equivalence no longer holds if $E$ is an infinite
graph. In this section, we extend the results of \cite{AAJZ1}, \cite{AAJZ2} to
obtain a complete characterization of and a structure theorem for Leavitt path
algebras over an arbitrary graph having finite Gelfand-Kirillov dimension. It
turns out that the "building blocks" for these algebras $L$ are von Neumann
regular rings and matrix rings over the Laurent polynomial ring $K[x,x^{-1}]$.
The acyclic vertices introduced in the previous section play a useful role.

We shall first recall the definition of the Gelfand-Kirillov dimension of
associative algebras over a field.

Let $A$ be a finitely generated algebra over a field $K$, generated by a
finite dimensional subspace $V=Ka_{1}\oplus\cdot\cdot\cdot\oplus Ka_{m}$. Let
$V^{0}=K$ and, for each $n\geq1$, let $V^{n}$ denote the $K$-subspace of $A$
spanned by all the mononomials of length $n$ in $a_{1},\cdot\cdot\cdot,a_{m}$.
Set $V_{n}=%
%TCIMACRO{\tsum \limits_{i=0}^{n}}%
%BeginExpansion
{\textstyle\sum\limits_{i=0}^{n}}
%EndExpansion
V^{i}$. Then the \textbf{Gelfand-Kirillov dimension} of $A$ ( for short, the
\textbf{GK-dimension} of $A$) is defined by%
\[
\text{GK-dim}(A):=\underset{n\rightarrow\infty}{\lim\sup}\log_{n}(\dim
V_{n}).
\]
It is known that the GK-dim($A$) is independent of the choice of the
generating subspace $V$.

If $A$ is an infinitely generated $K$-algebra, then the GK-dimension of $A$ is
defined as
\[
\text{GK-dim}(A):=\underset{B}{Sup\text{ }}\text{GK-dim}(B)
\]
where $B$ runs over all the finitely generated $K$-subalgebras of $A$.

Some useful examples the GK-dimension of algebras (see \cite{KL}) are: The
GK-dimension the matrix ring $M_{\Lambda}(K)$ is $0$ and the GK-dimension of
the matrix ring $M_{\Lambda^{\prime}}(K[x,x^{-1}])$ is $1$, where
$\Lambda,\Lambda^{\prime}$ are arbitrary index sets.

We also note that if an algebra $A$ has finite GK-dimension, then every
subalgebra of $A$ and every homomorphic image of $A$ also has finite
GK-dimension. But this does not hold for extensions: If $I$ is an ideal of an
algebra $A$, it may happen that both $I$ and $A/I$ have finite GK-dimension,
but $A$ has infinite GK-dimension. We refer to \cite{KL} for these and for
other properties and results on\ the GK-dimension of algebras.

As mentioned in the first paragraph of this section, we shall now give an
example of infinite graph $E$ in which distinct cycles have no common vertex,
but neither all the simple $L_{K}(E)$-modules are finitely presented nor the
GK-dimension of $L_{K}(E)$ is finite.

Example: Let $E=F\cup G$ be the union of two graphs $F,G$ together with a
connecting edge $g$. Specifically, $F$ is the graph obtained by removing the
vertex $w$ and edge $e$ from the graph $E^{\prime}$ mentioned at the end of
Section 3. Thus $F$ consists of infinitely many loops $c_{i}$ based at
vertices $v_{i}$ for $i=1,2,\cdot\cdot\cdot$ and, for each $i$, there is an
edge $e_{i}$ with $s(e_{i})=v_{i+1}$ and $r(e_{i})=v_{i}$. The graph $G$ is
the countable infinite clock, namely, $G^{0}=\{u\}\cup\{w_{i}:i=1,2,\cdot
\cdot\cdot\}$ and $G^{1}=\{f_{i}:i=1,2,\cdot\cdot\cdot\}$ such that, for all
$i$, $s(f_{i})=u$ and $r(f_{i})=w_{i}$. Finally, there is a connecting edge
$g$ with $s(g)=v_{1}$ and $r(g)=u$. Now clearly distinct cycles/loops in $E$
have no common vertex. Since $E$ is not row-finite, Corollay
\ref{row-finiteness} implies that not all simple modules over $L_{K}(E)$ are
finitely presented. Also $L_{K}(E)$ does not have finite GK-dimension. To see
this, let, for each $n>1$, $F_{n}$ denote the subgraph where $(F_{n}%
)^{0}=\{v_{1},\cdot\cdot\cdot,v_{n}\}$ and $(F_{n})^{1}=\{e_{1},\cdot
\cdot\cdot,e_{n-1}\}\cup(c_{1},\cdot\cdot\cdot,c_{n}\}$. Then $F_{n}$ is a
complete subgraph. If $B_{n}$ is the subalgebra generated by $F_{n}$, then by
Theorem 5 of \cite{AAJZ1}, $B_{n}\cong L_{K}(F_{n})$ has GK-dimension $2n-1$.
Consequently, GK-dim($L_{K}(E))\geq\sup\{GK-\dim(B_{n})\}$ is infinite.

If we consider just the graph $F$, then by Theorem \ref{Non Acyclic case}
every simple left/right module over $L_{K}(F)$ is finitely presented (and
distinct cycles in $E$ are disjoint), but by the preceding arguments
$L_{K}(F)$ does not have finite GK-dimension. On the other hand, since $G$ is
acyclic, $L_{K}(G)$ is a directed union of direct sums of matrix rings over
$K$ (see Theorem 1, \cite{AR}) and so has GK-dimension $0$, but not every
simple left/right module over $L_{K}(G)$ is finitely presented , by Theorem
\ref{No cycles}.

Our goal is to give complete description of the Leavitt path algebras over
arbitrary graphs having finite GK-dimension. To accomplish that, we shall be
using Result (a) from the previous section together with following Result (b).

\textbf{Result (b)}. The subalgebra construction using a\ finite set of edges
in a graph (\cite{AR}): Let $E$ be an arbitrary graph and let $F$ be a finite
set of edges in $E$. Then the graph $E_{F}$ is defined by setting

$(E_{F})^{0}=F\cup\lbrack(r(F)\cap s(F)\cap s(E^{1}\backslash F)]\cup
(r(F)\backslash s(F))$;

$(E_{F})^{1}=\{(e,f)\in F\times(E_{F})^{0}:r(e)=s(f)\}\cup$

$\{(e,r(e)):e\in F$ with $r(e)\in(r(F)\backslash s(F))\}$;

Here $r,s$ are defined by $s(x,y)=x$ and $r(x,y)=y$ for any $(x,y)\in
(E_{F})^{1}$.

A graded monomorphism $\theta:L_{K}(E_{F})\rightarrow L_{K}(E)$ was defined in
\cite{AR} and in Proposition 1 of that paper it was shown that $im(\theta)$
contains $F\cup F^{\ast}$ and $\{r(e):e\in F\}$.

The following observations will be useful in our proof: If $F$ is a finite set
of edges in $E$ then a path $p=e_{1}e_{2}\cdot\cdot\cdot e_{n}$ with $e_{i}\in
F$ is a cycle in $E$ if and only if $\bar{p}=(e_{1},e_{2})\cdot\cdot
\cdot(e_{n},e_{1})$ is a cycle in $E_{F}$. Moreover, if $C_{1}\geq\cdot
\cdot\cdot\geq C_{k}$ is a chain of cycles in $E$ where the edges in all of
the $C_{i}$ belong to $F$ then $\bar{C}_{i}\geq\cdot\cdot\cdot\geq\bar{C}_{k}$
is a chain of cycles in $E_{F}$. Also suppose $C_{1}=e_{1}e_{2}\cdot\cdot\cdot
e_{n}$ and $C_{2}=f_{1}f_{2}\cdot\cdot\cdot f_{m}$ are two cycles in $E$ with
$F^{\prime}=\{e_{1},\cdot\cdot\cdot,e_{n},f_{1},\cdot\cdot\cdot,f_{m}\}$ then
$C_{1},C_{2}$ will have a common vertex $s(e_{1})=s(f_{1})$ in $E$ if and only
if, in the graph $E_{F^{\prime}}$, $(e_{1},e_{2})\cdot\cdot\cdot(e_{n-1}%
,e_{n})(e_{n},f_{1})(f_{1},f_{2})\cdot\cdot\cdot(f_{m-1},f_{m})(f_{m},e_{1})$
is a cycle sharing common vertices with the cycle $(e_{1},e_{2})\cdot
\cdot\cdot(e_{n},e_{1})$ and with the cycle $(f_{1},f_{2})\cdot\cdot
\cdot(f_{m},f_{1})$.

We begin with the following easy Lemma.

\begin{lemma}
\label{Intersection Zero}Let $E$ be an arbitrary graph. If $A$ is the (graded)
ideal generated by the set $X$ of all the acyclic vertices in $E$ and $N$ is
the (graded) ideal generated by the set $Y$ of vertices in all the cycles with
no exits in $E$, then $A\cap N=0$.
\end{lemma}

\begin{proof}
Now both the ideals $A$ and $N$ possess local units and so $A\cap N=AN$. So it
is enough to show that $AN=0$. Suppose $a=%
%TCIMACRO{\tsum }%
%BeginExpansion
{\textstyle\sum}
%EndExpansion
k\alpha_{i}\beta_{i}^{\ast}\in A$ so that $r(\alpha_{i})=r(\beta_{i})\in X$
and let $b=%
%TCIMACRO{\tsum }%
%BeginExpansion
{\textstyle\sum}
%EndExpansion
l_{j}\gamma_{j}\delta_{j}^{\ast}\in N$ so that $r(\gamma_{j})=r(\delta_{j})\in
Y$. If $ab\neq0$, then for some $i,j$, $\alpha_{i}\beta_{i}^{\ast}\gamma
_{j}\delta_{j}^{\ast}\neq0$ which implies $\beta_{i}^{\ast}\gamma_{j}\neq0$.
This means that either $\beta_{i}=\gamma_{j}p$ or $\gamma_{j}=\beta_{i}q$
where $p,q$ are some paths. This leads to a contradiction for the following
reasons. If $\beta_{i}=\gamma_{j}p$, then $p$ gives rise to an exit for the no
exit cycle containing $r(\gamma_{j})$, a contradiction. If $\gamma_{j}%
=\beta_{i}q$, then $q$ is a path from $r(\beta_{i})$ to the vertex
$r(\gamma_{j})$ which sits on a cycle, contradicting the fact that
$r(\beta_{i})$ is an acyclic vertex. Hence $A\cap N=AN=0$.
\end{proof}

We are now ready to prove the main theorem of this section.

\begin{theorem}
\label{Finite GK-dim}Let $E$ be an arbitrary graph, $K$ be any field and let
$L=L_{K}(E)$. Then the following conditions are equivalent:

(i) $\ \ L$ has finite GK-dimension $\leq m$, where $m$ is a nonnegative integer;

(ii) \ The relation $\geq$ defines a partial order in the set $P$ of all
\ cycles in $E$ and there is a non-negative integer $m$ such that every chain
in $P$ has length at most $m$;

(iii) $L$ is the union of a finite chain of graded ideals%
\[
0\leq I_{0}<I_{1}<\cdot\cdot\cdot<I_{m}=L
\]
where $m$ is a fixed non-negative integer, $I_{0}$ (may be zero) is von
Neumann regular and, for each $j$ \ with $0\leq j\leq m-1$, $I_{j+1}/I_{j}$ is
a direct sum of a von Neumann regular ring and/or direct sums of matrix rings
of the form $M_{\Lambda_{j}}(K[x,x^{-1}])$ where $\Lambda_{j}$ are arbitrary
index sets.
\end{theorem}

\begin{proof}
Assume (i). Suppose, by way of contradiction, $E$ contains two distinct cycles
$C,C^{\prime}$ having a common vertex. Let $F$ be the set of all edges
belonging to $C$ and $C^{\prime}$. Then the (finite) graph $E_{F}$ will
contain two cycles with a common vertex and so, by Theorem 5 of \cite{AAJZ1},
$L_{K}(E_{F})$ has infinite GK-dimension. Then the subalgebra $im(\theta
)\cong$ $L_{K}(E_{F})$ (See Result (b)) and hence $L$ will have infinite
GK-dimension, a contradiction. So distinct cycles in $E$ have no common
vertex. \ This makes the relation $\geq$ antisymmetric and hence a partial
order in the set of all the cycles in $E$. If $m=GK-dim(L)=0$, then $E$ must
be acyclic. Because, if there is a cycle $c$ in $E$ based at a vertex $v$,
then consider $V:=Kv\oplus Kc$. Since $\dim(V^{n})\geq n$, this forces
$GK-\dim(L)\geq1$, a contradiction. Thus $E$, being acyclic, trivially
satisfies Condition (ii). So Assume $m\geq1$. Suppose there is a chain of
cycles $C_{1}\geq\cdot\cdot\cdot\geq C_{d}$ of length $d>m$ in $E$. For each
$i$, let $\gamma_{i}$ be a path connecting a fixed vertex on $C_{i}$ to a
fixed vertex on $C_{i+1}$. If $F$ is the set of all the edges in the cycles
$C_{1},\cdot\cdot\cdot,C_{d}$ and in the paths $\gamma_{1},\cdot\cdot
\cdot\cdot,\gamma_{d-1}$, then the subalgebra $im(\theta)\cong L_{K}(E_{F})$
will have, by Theorem 5 of \cite{AAJZ1}, GK-dimension $\geq2d-1$ which is
%TCIMACRO{\TEXTsymbol{>} }%
%BeginExpansion
$>$
%EndExpansion
$m$. Clearly then the GK-dimension of $L$ is$\ >m$, a \ contradiction. This
proves (ii).

Assume (ii). We prove (iii) by induction on $m$. Suppose $m=0$. This means
that the graph $E$ contains no cycles and consequently $L$ is von Neumann
regular, by Theorem 1 of \cite{AR}. So Condition (iii) holds with $L=I_{0}$.
Suppose $m\geq1$ and assume that we have shown that $L$ satisfies Condition
(iii) if the upper bound for the lengths of chains of cycles in $E$ is $m-1$.
Note that if a cycle $C$ is a minimal element in the (artinian) partially
ordered set $P$ of cycles under $\geq$ in $E$, then either $C$ has no exits or
for each exit $e$ for $C$, $r(e)$ is an acyclic vertex. Let $I_{0}$ be the
graded ideal generated by the set $\{r(e):e$ an exit for a minimal cycle in
$P\}$. Since each such $r(e)$ is an acyclic vertex, $I_{0}$ is von Neumann
regular, by Corollary \ref{acyclic vertex implies von neumann regular}. If
$H_{0}=I_{0}\cap E^{0}$, then $I_{0}=I(H_{0},\emptyset)$ where $\emptyset$ is
the empty set and in $E\backslash(H_{0},\emptyset)$ the minimal cycles in the
poset of cycles have no exit. Now $L/I_{0}\cong L_{K}(E\backslash
(H_{0},\emptyset)$. Identifying $L/I_{0}$ with $L_{K}(E\backslash
(H_{0},\emptyset)$, let $I_{1}/I_{0}$ denote the (graded) ideal generated by
all the acyclic vertices and the vertices in all the cycles without exits in
$E\backslash(H_{0},\emptyset)$. By Proposition 3.7 of \cite{AAPS} and further
Corollary \ref{acyclic vertex implies von neumann regular} and Lemma
\ref{Intersection Zero} above, $I_{1}/I_{0}$ is a direct sum of a von Neumann
regular ring and a direct sum of matrix rings of the form $M_{\Lambda^{(1)}%
}(K[x,x^{-1}])$ with $\Lambda^{(1)}$ arbitrary index sets. Let $H_{1}%
=I_{1}\cap E^{0}$. Then $L/I_{1}\cong L_{K}(E\backslash(H_{1},\emptyset))$.
Since $(E\backslash(H_{1},\emptyset))^{0}=E^{0}\backslash H_{1}\cup
\{u^{\prime}:u\in B_{H_{1}}\}$ and since the $u^{\prime}$ are all sinks in
$E\backslash(H_{1},\emptyset)$, the maximum length of chains in $E\backslash
(H_{1},\emptyset)$ is $m-1$. So by induction, $L/I_{1}$ is the union of a
chain of graded ideals which we conveniently write as
\[
0<I_{2}/I_{1}<\cdot\cdot\cdot<I_{m}/I_{1}=L/I_{1}%
\]
where, for $1\leq j\leq m-1$, $(I_{j+1}/I_{1})/(I_{j}/I_{1})$ is a direct sum
of a von Neumann regular ring and direct sums of matrix rings of the form
$M_{\Lambda^{(j)}}(K[x,x^{-1}])$ with $\Lambda^{(j)}$ arbitrary index sets.
From this we immediately obtain the needed chain for $L$ with the stated
properties of Condition (iii).

Assume (iii). We wish to prove (ii). For each $j$, let $H_{j}=I_{j}\cap E^{0}$
and so $E^{0}$ is the union of the chain of hereditary saturated sets
$\emptyset\subseteqq H_{0}\subset H_{1}\subset\cdot\cdot\cdot\subset
H_{m}=E^{0}$. Suppose, by way of contradiction, $E$ contains two cycles $g,h$
having a common vertex $v$. Let $j\geq0$ be the smallest integer such that
$v\notin H_{j}$. Then $v\in H_{j+1}$ and in that case $g^{0},h^{0}\subset
H_{j+1}\backslash H_{j}$. Now $I_{j+1}/I_{j}$ is a direct sum a von Neumann
regular ring and direct sums of matrix rings of the form $M_{\Lambda^{(j)}%
}(K[x,x^{-1}])$ with $\Lambda^{(j)}$ arbitrary index sets and so
$I_{j+1}/I_{j}$ has finite GK-dimension. Also identifying $L/I_{j}$ with
$L_{K}(E\backslash(H_{j},\emptyset))$, we get an isomorphism $I_{j+1}%
/I_{j}\cong L_{K}((\overline{E}(H_{j+1}\backslash H_{j},\emptyset))$ by
(Theorem 6.1, \cite{RT}). From the proof of (i) $\Longrightarrow$ (ii) it is
then clear that distinct cycles in $\overline{E}(H_{j+1}\backslash
H_{j},\emptyset)$ have no common vertex. But this contradicts the assumption
that the cycles $g,h$ have a common vertex $v$ in $\overline{E}(H_{j+1}%
\backslash H_{j},\emptyset)$. Hence no two distinct cycles in $E$ will have a
common vertex. Suppose there is a chain of cycles $C_{1}\geq C_{2}\geq
\cdot\cdot\cdot\geq C_{d}$ with $d>m$ and the cycle $C_{d}$ is based at a
vertex $u$. As before if $j\geq0$ is the smallest integer such that $u\notin
I_{j}$, then $u\in H_{j+1}\backslash H_{j}$ and so $(C_{d})^{0}\subset
H_{j+1}\backslash H_{j}$. If also $(C_{d-1})^{0}\subset H_{j+1}\backslash
H_{j}$, then the GK-dimension of $I_{j+1}/I_{j}$ will be $\geq2$. Because, if
$F$ is the finite set of edges on the cycles $C_{d},C_{d-1}$ and on a path
connecting these two cycles, then the subalgebra of $I_{j+1}/I_{j}$ isomorphic
to $L_{K}(_{F}E)$ will have GK-dimension $\geq2$, by Theorem 5 of
\cite{AAJZ1}, contradicting the fact that $I_{j+1}/I_{j}$ has GK-dimension
$\leq1$. So $(C_{d-1})^{0}\notin H_{j+1}$. Proceeding like this we reach a
conclusion that $(C_{d-(m-j)})^{0}\notin H_{j+(m-j)}=H_{m}=E^{0}$, a
contradiction. Hence every chain of cycles \ in $E$ must have length \ at most
$m$ and this proves (ii).

Assume (ii). It was shown in (Proposition 2, \cite{AR})
$L=\underset{\longrightarrow}{Lim}B(S)$ where $S$ varies over all the finite
subsets of $L$ and $B(S)$ is a subalgebra generated by $S$. Moreover, each
$B(S)$ is isomorphic to $L_{K}(E_{F})\oplus V$ where $F$ is a finite set of
edges on the paths that show up in the representation of the elements of $S$
as $K$-linear combinations of monomials and $V$ is a finite dimensional
$K$-algebra (see Proposition 2, \cite{AR}). By hypothesis and the remarks in
Result (b) above, the graph $E_{F}$, for every finite set of edges $F$,
satisfies Condition (ii). Hence, by Theorem 5 of \cite{AAJZ1}, each
$L_{K}(E_{F})$ has GK-dimension $\leq m$. This means each $B(S)$ and hence $L$
has GK-dimension $\leq m$. This proves (i).
\end{proof}

REMARK: We wish to point out that, for a finite graph $E$, the chain of ideals
in Condition (3) of Theorem \ref{Non Acyclic case} becomes finite with the
succesive quotients direct sums of matrix rings over $K$ and/or $K[x,x^{-1}]$
and this (for the finite graph $E$) has been shown in \cite{AAJZ2} to be a
necessary condition for $L_{K}(E)$ to have finite GK-dimension and that in
\cite{R}, it is also shown to be a sufficient condition for finite
GK-dimension of $L_{K}(E)$.

\begin{acknowledgement}
I am deeply grateful to Gene Abrams for giving me the benefit his preprint
\cite{AMT} containing the crucial Lemmas that were used in the proof of
Corollary \ref{Irrational => no fp} and for useful discussions.
\end{acknowledgement}

\end{document}